\documentclass{ifacconf}

\usepackage{natbib}        % required for bibliography
\usepackage{bm}
\usepackage{amsmath}
\usepackage{amsfonts}
\usepackage{amssymb}
\usepackage{import}
\usepackage{graphicx}
\usepackage{color}
\usepackage{comment}
\makeatletter
\providecommand{\hyper@nopatch@sectioning}{}
\makeatother
\usepackage{hyperref}
\makeatletter
\AtBeginDocument{%
}
\makeatother
\newcommand{\IdentityMatrix}[1]{\bm{I}}
\newcommand{\ZerosMatrix}[2]{\ensuremath{\bm{0}}}
\newcommand{\Expect}[1]{\mathcal{E}\left\{#1\right\}}

\newtheorem{lemma}{Lemma}
\newtheorem{theorem}{Theorem}
\newtheorem{remark}{Remark}

\newcommand{\GR}[1]{\textcolor{black}{#1}}

\newtheorem{proof}{Proof}

\allowdisplaybreaks
\begin{document}
\begin{frontmatter}

\title{Stochastic Robust Linear {$\mathcal{W}_\infty$} Control via Dynamic Output Feedback\thanksref{footnoteinfo}} 

\thanks[footnoteinfo]{This work was supported in part by the  CAPES through the Academic Excellence Program (PROEX), CNPq under Grants 317058/2023-1 and 422143/2023-5, FAPEMIG under Grant BPD-00960-22, and in part by Petrobras/ANP under Grants 2023/00494-5 and 2023/00643-0.}

\author[First]{Daniel Neri Cardoso} 
\author[First,Second]{Guilherme Vianna Raffo} 

\address[First]{Graduate Program in Electrical Engineering - Universidade Federal de Minas Gerais - Av. Antônio Carlos 6627, 31270-901, Belo Horizonte, MG, Brazil (e-mail: \{danielneri, raffo\}@ufmg.br)}
\address[Second]{Department of Electronic Engineering - Universidade Federal de Minas Gerais, Belo Horizonte, MG, Brazil}

\begin{abstract}               
%This paper introduces a robust $\mathcal{W}_\infty$ optimal control framework for linear Itô stochastic systems using a weighted Sobolev-space performance measure defined on the expected state and its weak derivative. The associated Hamilton–Jacobi equation is derived, and an equivalent LMI-based semidefinite programming method is proposed for dynamic output-feedback synthesis. A rigorous stability analysis ensures mean-square ultimate boundedness with minimized ultimate bound. A numerical example illustrates the effectiveness of the proposed method.
%{ESTENDER UM POUCO O ABSTRACT}
This paper introduces a robust $\mathcal{W}_\infty$ optimal control framework for linear Itô diffusions using a weighted Sobolev-space performance measure. Because the sample paths of Itô diffusions are nondifferentiable, the formulation leverages the weak derivative of the expected state. An LMI-based semidefinite program is developed for dynamic output-feedback synthesis, and a rigorous stability analysis guarantees mean-square ultimate boundedness with minimized ultimate bound. A numerical example demonstrates that the proposed approach provides effective disturbance attenuation with fast transient performance.
\end{abstract}

\begin{keyword}
Optimal Control, Stochastic Control, Robust Control, Dynamic Output Feedback, {$\mathcal{W}_\infty$} Control, Sobolev-Space-Based Performance Measures.
\end{keyword}

\end{frontmatter}
%===============================================================================

\vspace{-2mm}
\section{Introduction}
\vspace{-2mm}

Since the late 1990s, the control systems \GR{field} has \GR{seen} rising interest in stochastic controllers. Unlike deterministic models, stochastic \GR{systems} incorporate intrinsic randomness \GR{to more realistically represent} real-world \GR{dynamics, motivating} extensive research on robust optimal control. Early contributions \GR{have included} mixed $\mathcal{H}_2/\mathcal{H}_\infty$ control for linear stochastic differential systems \GR{\citep{peters1994mixed} and robust $\mathcal{H}_\infty$ controllers under stochastic parameter uncertainty \citep{ugrinovskii1998robust}}. Extensions to nonlinear settings with state-dependent \GR{Itô} noise have been explored in \cite{zhang2006state} within a mixed $\mathcal{H}_2/\mathcal{H}_\infty$ framework\GR{, while \cite{zhu2018mixed}} recently have proposed a mixed $\mathcal{H}_\infty$–passivity design for stochastic nonlinear systems under aperiodic sampling.

\GR{Robust} optimal controllers is traditionally grounded in $\mathcal{L}_2$-space, \GR{where the} $\mathcal{L}_2$-norm serves as a measure of signal energy. \GR{Within} this energy-based \GR{framework}, $\mathcal{H}_\infty$ theory is widely adopted \GR{for its ability to attenuate} exogenous disturbances below a predefined level $\gamma_{\mathcal{H}_\infty}$ while minimizing a specified cost variable under worst-case scenarios \citep{van1992sub}. This makes $\mathcal{H}_\infty$ control particularly \GR{effective for systems with intrinsic uncertainties and unknown disturbances, with applications ranging from} biological \GR{and power systems} \citep{wu2011multiobjective,sedhom2020multistage} \GR{to} macroeconomic models \citep{ma2012infinite}.

Despite its advantages, classical $\mathcal{H}_\infty$ control focuses on signal energy \GR{rather than} the rate of variation\GR{, which} often \GR{results in} oscillatory transients~\citep{chilali1996h}. \GR{While weighting input energy} in the cost variable \GR{is a common workaround, it can compromise} performance, \GR{as effective} disturbance attenuation \GR{frequently requires} additional control \GR{effort. Furthermore}, in multiple-input multiple-output (MIMO) systems, the \GR{relationship between} individual input \GR{weights} and the cost variable\GR{'s rate of variation} is \GR{often non-trivial} to determine.

To address this, \cite{aliyu2011extending} have introduced a Sobolev-space ({spaces} of functions in $\mathcal{L}_p$ {space} whose weak derivatives up to order $m$ also lie in $\mathcal{L}_p$ {space}~\citep{treves2016topological}) performance measure that incorporates the cost variable\GR{'s rate of variation}. This $\mathcal{W}_\infty$ {framework} \GR{has been} extended to weighted Sobolev spaces in~\cite{AUT2019}, where comparative {analyses have} demonstrated \GR{superior} transient behavior and faster disturbance attenuation \GR{than} $\mathcal{L}_2$-based formulations. However, existing $\mathcal{W}_\infty$ developments remain restricted to deterministic systems. In stochastic settings, the sample paths of Itô diffusions are almost surely non-differentiable~\citep{shreve2004stochastic}, \GR{preventing the direct application of weak derivatives and the standard $\mathcal{W}_\infty$ approach.}

Motivated by \GR{these} gaps, this work \GR{demonstrates} that \GR{while the Itô diffusion} sample paths are almost surely non-differentiable, the derivative of their expected values exists and can be treated as a weak derivative. This enables a consistent extension of the $\mathcal{W}_\infty$ framework to stochastic linear systems with bounded diffusion coefficients. \GR{We} propose the stochastic $\mathcal{W}_\infty$ {optimal control problem} (OCP) and derive a dynamic output-feedback controller \GR{for systems with} partial state availability. The OCP is formulated \GR{using the} {Hamilton-Jacobi-Bellman-Isaacs} equation and \GR{resolved via} semidefinite program (SDP) \GR{for} efficient synthesis. \GR{Finally,} stability analysis establishes mean-square ultimate boundedness and \GR{proves} that the proposed controller minimizes the corresponding ultimate bound.

The contributions of this work are twofold: (i) we formulate a stochastic robust $\mathcal{W}_\infty$ OCP for systems \GR{with bounded} diffusion coefficients and provide a stability analysis for the resulting dynamic output-feedback controllers; and (ii) we derive {linear matrix inequality} (LMI) conditions that recast the OCP \GR{as} a semidefinite program (SDP), enabling the straightforward synthesis of stochastic robust $\mathcal{W}_\infty$ controllers for linear systems.
\vspace{-1mm}

\subsubsection{Notation and definitions:} Italic lowercase letters denote scalars, boldface italic lowercase letters denote vectors, and boldface italic uppercase letters denote matrices. The superscripts {\small$(\cdot)'$\normalsize} and {\small$(\cdot)^{-1}$\normalsize} denote the transpose and inverse operations, respectively. \small$\mathbb{N} \triangleq \{1,2,...\}$\normalsize, \small$\mathbb{R} \triangleq (-\infty, \infty)$, $\mathbb{R}_{\geq 0} \triangleq [0, \infty )$\normalsize, \small$\mathbb{R}^n \triangleq \{\bm{r} = [r_1 \; ... \; r_n]':r_i \in \mathbb{R}\}$\normalsize, and \small$\mathbb{R}^{n\times m} \triangleq \{\bm{R} = [\bm{r}_1 \;... \;\bm{r}_m]:\bm{r}_i \in \mathbb{R}^n, i \in \{1, 2, \cdots, m\}\}$\normalsize. \small$\ZerosMatrix{n}{m}$ \normalsize and \small$\IdentityMatrix{n}$ \normalsize are, respectively, zero and identity matrices with appropriate dimensions. Let \small$t \in \mathbb{R}_{\geq 0}$ \normalsize denote the time variable and \small$v(\bm{x},t): \mathbb{R}^{n_x}\times \mathbb{R}_{\geq 0} \to \mathbb{R}$\normalsize, then \small$v_x ={\partial v(\bm{x},t)}/{\partial \bm{x}}$ \normalsize and \small$v_t ={\partial v(\bm{x},t)}/{\partial t}$\normalsize. Let \small$\bm{z}(t): \mathbb{R}_{\geq 0} \rightarrow \mathbb{R}^{n_z}$ \normalsize be a time-varying function, then \small$\dot{\bm{z}}(t) \triangleq {d\bm{z}(t)}/{dt}$ \normalsize denotes its first time derivative.
Let \small$p \in \mathbb{N}\cup\{\infty\}$, $m \in \mathbb{N}$, and $\bm{z}: \Psi \to \mathbb{R}^{n_z}$\normalsize. If \small$\bm{z}$ \normalsize belongs to the weighted Lebesgue space, i.e. \small$\bm{z} \in \mathcal{L}_{p,\bm{\Lambda}}[\Psi]$\normalsize, then its weighted \small$\mathcal{L}_{p}$\normalsize-norm is finite, \small$||\bm{z}||_{\mathcal{L}_{p,\bm{\Lambda}}} \triangleq \left(\int_{\Psi} ||\bm{\Lambda}^{1/p}\bm{z}||_p^p~d\Psi\right)^{{1}/{p}} < \infty$\normalsize, where \small$\bm{\Lambda}$ \normalsize is a positive definite symmetric matrix with appropriate dimension. If \small$\bm{z}$ \normalsize belongs to the weighted Sobolev space, i.e. \small$\bm{z} \in \mathcal{W}_{m,p,\bm{\Gamma}}[\Psi]$\normalsize, then \small$||\bm{z}||_{\mathcal{W}_{m,p,\bm{\Gamma}}} \triangleq \big(\sum_{\alpha = 0}^{m} ||{\partial^{\alpha}\bm{z}}||_{\mathcal{L}_{p,\bm{\Gamma}}}^p\big)^{{1}/{p}} < \infty$\normalsize,  with \small$\bm{\Gamma}  \triangleq \{\bm{\Gamma}_0, ..., \bm{\Gamma}_m\}$\normalsize, where \small${\partial^{\alpha}\bm{z}}$ \normalsize is the \small$\alpha$\normalsize-th weak derivative of \small$\bm{z}$\normalsize. It is well known that the weak derivative coincides with the standard derivative when the latter exists. The definitions of the Lebesgue \small$\mathcal{L}_{p}$ \normalsize and Sobolev \small$\mathcal{W}_{m,p}$ \normalsize spaces follow straightforwardly considering identity matrices as weights.
Let \small$\left(\Omega,\mathcal{F},\{\mathcal{F}_t\}_{t\geq 0},\mathbb{P}\right)$ \normalsize be a complete filtered probability space, where \small$\Omega$ \normalsize is a nonempty set, \small$\mathcal{F}$ \normalsize is a \small$\sigma$\normalsize-algebra of the subsets of \small$\Omega$\normalsize, \small$\{\mathcal{F}_t\}_{t\geq 0}$ \normalsize is a filtration, and \small$\mathbb{P}$ \normalsize is a probability measure \citep{shreve2004stochastic}. The standard Brownian motion \small$m(t) \in \Omega$\normalsize, such that \small$m(t): \mathbb{R}_{\geq 0} \to \mathbb{R}$\normalsize, is adapted to the filtration \small$\{\mathcal{F}_t\}_{t\geq 0}$ \normalsize and is a martingale with respect to this filtration. It has independent increments and satisfies \small$m(0) {=} 0$\normalsize{,} \small$\left(m(t_{f}) {-} m(t_o)\right) {\sim} \mathcal{N}\left(0,t_{f} {-} t_o\right),{~} \forall \left(t_{f} > t_o\right)$\normalsize. If \small$\bm{z}(t)$ \normalsize is \small$\mathcal{F}_t$\normalsize-measurable, then the information available at  {\small$t$\normalsize} is sufficient to evaluate {\small$\bm{z}(t)$\normalsize} at that time. If \small$\bm{z}(t) {\in} \mathcal{L}^{\mathcal{F}_t}_p[0,\infty)$\normalsize, then \small$\bm{z}(t)$ \normalsize is \small$\mathcal{F}_t$\normalsize-measurable and \small$\Expect{||\bm{z}(t)||_{\mathcal{L}_p}} {<} \infty$\normalsize. Finally, \small$\Expect{\bm{z}(t)|\mathcal{F}_t}$ \normalsize represents the conditional expectation operator given the information available up to \small$t$\normalsize, denoted by \small$\mathcal{F}_t$\normalsize.
\vspace{-2mm}
\section{Preliminaries}
\vspace{-2mm}
\label{Preliminaries}

Consider the nonlinear time-varying {dynamical} system
\begin{align}
\label{GenericNonlinearSystem}
\mathcal{P}_1:\begin{cases}
    \dot{\bm{x}}(t) &= f(\bm{x},\bm{u},\bm{w},{t}), \\
    \bm{z}(t) &= h(\bm{x},\bm{u}),~~\bm{x}(0) = \bm{x}_0,
\end{cases}
\end{align}
where $t \in \mathbb{R}_{\geq 0}$, $\bm{x}: \mathbb{R}_{\geq 0} \to \mathbb{R}^{n_x}$ is the state vector, $\bm{u}: \mathbb{R}_{\geq 0} \to \mathbb{R}^{n_u}$ is the control input, $\bm{w}(t): \mathbb{R}_{\geq 0} \to \mathbb{R}^{n_w}$ is the disturbance, and $\bm{z}(t): \mathbb{R}_{\geq 0} \to \mathbb{R}^{n_z}$ is the cost variable. 
The classic nonlinear $\mathcal{H}_\infty$ control technique aims to find \GR{an admissible} control law $\bm{u}^*$ \GR{solving}
\begin{align}
   \bm{u}^* =  \arg\min_{\bm{u} \in \mathcal{U}}\max_{\bm{w} \in \mathcal{W}} \left(||\bm{z}(t)||^2_{\mathcal{L}_{2}} - \gamma_{\mathcal{H}}^2 ||\bm{w}(t)||^2_{\mathcal{L}_2}\right).
\end{align}
%where $ \mathcal{U}$ is the space of admissible control inputs, and $\mathcal{W} =\mathcal{L}_2[\mathbb{R}_{\geq 0})$.
If it exists, $\bm{u}^*$ ensures the closed-loop system is $\mathcal{L}_2$-stable with $\mathcal{L}_2$-gain $\gamma_{\mathcal{H}} \in \mathbb{R}_{> 0}$ \citep{van1992sub}, guaranteeing
    \small$||\bm{z}(t)||^2_{\mathcal{L}_{2}} \leq \gamma_{\mathcal{H}}^2 ||\bm{w}(t)||^2_{\mathcal{L}_2}+ c_1$, \normalsize
for all $\bm{w} \in \mathcal{W} {=} \mathcal{L}_2[\mathbb{R}_{> 0}]$, \GR{where} $c_1\in \mathbb{R}_{> 0}$ \GR{accounts for the energy contribution of initial conditions. Minimizing} $\gamma_{\mathcal{H}} \in \mathbb{R}_{>0}$ is \GR{desired as it represents the worst-case energy gain from disturbances to the cost variable. However,} $\mathcal{H}_\infty$ control \GR{neglects} the influence of the cost variable\GR{'s rate of change, motivating} the $\mathcal{W}_\infty$ optimal control \GR{framework}.

\GR{The} $\mathcal{W}_\infty$ framework \GR{incorporates dynamics via} a weighted Sobolev norm\GR{.} The optimal control law solves
%\vspace{-1mm}
\begin{align}
\label{OptimalWinfControlProblem}
\bm{u}^* {=} \arg\min_{\bm{u} \in \mathcal{U}} \max_{\bm{w} \in \mathcal{W}} \big(||\bm{z}(t)||^2_{\mathcal{W}_{1,2,\bm{\Gamma}}} {-} \gamma_{\mathcal{W}}^2 ||\bm{w}(t)||^2_{\mathcal{L}_2}\big),
\end{align}
ensuring $\mathcal{W}_{1,2,\Gamma}$-stability with gain $\gamma_{\mathcal{W}} \in \mathbb{R}_{\geq 0}$ \citep{AUT2019}, such that $||\bm{z}(t)||^2_{\mathcal{W}_{1,2,\bm{\Gamma}}} \leq \gamma_{\mathcal{W}}^2 ||\bm{w}(t)||^2_{\mathcal{L}_2} + c_2$ for all $\bm{w} \in \mathcal{W}$, $c_2 > 0$. Minimizing $\gamma_{\mathcal{W}}$ accelerates disturbance attenuation and yields smoother transients \GR{with reduced} overshoot, \GR{as} disturbances affect \GR{signal} derivatives before the cost variable itself. However, \GR{stochastic processes} are almost surely non-differentiable\GR{, lacking} weak derivatives. To resolve this, we propose \GR{a} stochastic $\mathcal{W}_\infty$ control \GR{formulation based on} the expected value of the cost variable.

\vspace{-2mm}
\section{Stochastic Robust Linear $\mathcal{W}_\infty$ Control}
\vspace{-2mm}
\label{RobustWStochasticApproach}

To establish the stochastic robust linear $\mathcal{W}_\infty$ OCP, consider the following Itô type stochastic differential equation:
\begin{align}
    \label{GenericStochasticSystem}
    \hspace{-2mm}\mathcal{P}_2{:}\begin{cases}
            d\bm{x}(t) = \bm{f}(\bm{x},\bm{u},\bm{w}) dt + \bm{E}_x(t) dm(t), \\
            \bm{y}(t) = \bm{A}_y\bm{x} + \bm{E}_y d\delta(t) \\
     \bm{z}(t) = \mathcal{E}\{\bm{x}(t)\}, \\
     \mathcal{E}\{\bm{x}(0)\} = \bm{x}_0,
    \end{cases}
\end{align}
where $\bm{x}_0 \in \mathbb{R}^{n_x}$, $m(t): \mathbb{R}_{\geq 0} \mapsto \mathbb{R}^{n_m}$ and  $\delta(t): \mathbb{R}_{\geq 0} \mapsto \mathbb{R}^{n_\delta}$ represent vectors whose elements are standard Brownian motion defined on the filtered probability space $\left(\Omega,\mathcal{F},\{\mathcal{F}_t\}_{t\geq 0},\mathcal{P}\right)$. Furthermore, $\bm{f}(\bm{x},\bm{u},\bm{w}) \triangleq \bm{A}\bm{x} + \bm{B}\bm{u} + \bm{D}\bm{w}$ is the drift coefficient, in which $\bm{x}(t)$, $\bm{u}(t)$, and $\bm{w}(t)$ are $\mathcal{F}_t$-measurable variables that follow the same definition as given in \eqref{GenericNonlinearSystem}, $\bm{E}_x(t): \mathbb{R}_{\geq 0} \mapsto \mathbb{R}^{n_x \times n_m}$ is the diffusion coefficient, $\bm{y}(t): \mathbb{R}_{\geq 0} \mapsto \mathbb{R}^{n_y}$ is the vector of measured variables, $\bm{A}_y  \in\mathbb{R}^{n_y\times n_x}$, and $\bm{E}_y \in \mathbb{R}^{n_y\times n_\delta}$.

\begin{remark}
    Given the stochastic system described in~\eqref{GenericStochasticSystem}, the state trajectories $\bm{x}(t)$ are, in general, non-differentiable and do not possess weak derivatives. To overcome this limitation, our analysis focuses on the expected value of the state. We formulate the control problem to ensure that $\Expect{\bm{x}(t)} \in \mathcal{W}_{1,2,\bm{\Gamma}}[0,\infty)$. 
\end{remark}

%\vspace{-2mm}
Taking into account \eqref{GenericStochasticSystem}, the stochastic robust $\mathcal{W}_\infty$ {control law} aims to achieve
%\vspace{-1mm}
\begin{align}
\label{StochasticWinfController}
    ||\bm{z}(t)||^2_{\mathcal{W}_{1,2,\bm{\Gamma}}} \leq \gamma_{\mathcal{W}}^2 ||\bm{w}(t)||^2_{\mathcal{L}_2}  + c_3, 
\end{align}
%\vspace{-1mm}
for all $\bm{w} \in \mathcal{L}^{\mathcal{F}_t}_2[0,\infty)$, $c_3\in \mathbb{R}_{> 0}$, with a given, sufficiently large, $\gamma_{\mathcal{W}} \in \mathbb{R}_{\geq0}$, where $\bm{\Gamma}  \triangleq \{\bm{\Gamma}_0, \bm{\Gamma}_1\}$ {is a} set of positive definite tuning matrices. 

\begin{lemma}
     Suppose that system \eqref{GenericStochasticSystem} satisfies the inequality \eqref{StochasticWinfController} and that $\bm{w}(t) \in \mathcal{L}^{\mathcal{F}_t}_2[0,\infty)$. Under these conditions, the system is guaranteed to achieve mean-square stability.
\end{lemma}
%\vspace{-1mm}
\begin{proof}
If $\bm{w}(t) \in \mathcal{L}^{\mathcal{F}_t}_2[0,\infty)$ and inequality \eqref{StochasticWinfController} is satisfied, then, $\bm{z}(t) \in \mathcal{W}_{1,2,\bm{\Gamma}}[0,\infty)$. Since the weighted Sobolev space is a subspace of the Lebesgue space $\mathcal{L}_2[\mathbb{R}_{> 0}]$, this leads to the following implications: $\bm{z}(t) \in \mathcal{W}_{1,2,\bm{\Gamma}}[0,\infty) \implies \bm{z}(t) \in \mathcal{L}_2[\mathbb{R}_{> 0}]$, thereby ensuring the mean-square stability of the closed-loop system~\citep{zhang2017stochastic}. $\qquad\qquad\qquad\qquad\quad\qquad\qquad\qquad\qquad\;\;\;$\qed
\end{proof}

\begin{lemma}
\label{TheoremSobolevspace}
If a solution exists for the OCP
%\vspace{-1mm}
\begin{align}
\label{WinfControlProblem2}
    \hspace{-1mm}\bm{u}^* {=} \arg\min_{\bm{u}\in \mathcal{U}}\max_{\bm{w} \in \mathcal{W}} \big(||\bm{z}(t)||^2_{\mathcal{W}_{1,2,\bm{\Gamma}}} - \gamma_{\mathcal{W}}^2 ||\bm{w}(t)||^2_{\mathcal{L}_2} \big),
\end{align}
for all $\bm{w} \in \mathcal{L}^{\mathcal{F}_t}_2[0,\infty)$, then the stochastic linear system \eqref{GenericStochasticSystem}, with the optimal control $\bm{u}^*$, satisfies  \eqref{StochasticWinfController}.
\end{lemma}
\begin{proof}
    Under the assumption of existence of a solution for \eqref{WinfControlProblem2}, then there also exists an optimal cost $v_{\mathcal{W}_\infty} \in \mathbb{R}$, such that 
    %\vspace{-2mm}
    \begin{align}
    \label{OptimalCost}
        v_{\mathcal{W}_\infty} = \min_{\bm{u}\in \mathcal{U}}\max_{\bm{w} \in \mathcal{W}} ||\bm{z}(t)||^2_{\mathcal{W}_{1,2,\bm{\Gamma}}} - \gamma_{\mathcal{W}}^2 ||\bm{w}(t)||^2_{\mathcal{L}_2} .
    \end{align}
    By denoting $\bm{z}^*$ and $\bm{w}^*$ as the min-max optimized values of $\bm{z}(t)$ and $\bm{w}(t)$, respectively, equation \eqref{OptimalCost} results in
    %\vspace{-1mm}
    \begin{align}
    \label{OptimalCost2}
        v_{\mathcal{W}_\infty} = ||\bm{z}^*||^2_{\mathcal{W}_{1,2,\bm{\Gamma}}} - \gamma_{\mathcal{W}}^2||\bm{w}^*||^2_{\mathcal{L}_2}.
    \end{align}
    However, if one considers an unknown disturbance $\bm{w}(t) \in \mathcal{L}^{\mathcal{F}_t}_2[0,\infty)$ instead of the worst-case scenario, then we have
   % \vspace{-2mm}
    \begin{align}
    \label{OptimalCost3}
        v_{\mathcal{W}_\infty} \geq ||\bm{z}^*||^2_{\mathcal{W}_{1,2,\bm{\Gamma}}} - \gamma_{\mathcal{W}}^2 ||\bm{w}(t)||^2_{\mathcal{L}_2},
    \end{align}
    which is equivalent to \eqref{StochasticWinfController}, concluding the proof. $\qquad$\qed
\end{proof}

%\vspace{-2mm}
The formulation of the stochastic robust $\mathcal{W}_\infty$ {optimal control problem} requires the existence of the weak derivative of the cost variable. The following lemma establishes the weak derivative of $\bm{z}(t)$.

\begin{lemma}
\label{TheoremWeakDerivative}
    Consider the stochastic system governed by the Itô differential equation \eqref{GenericStochasticSystem}.  Then, the drift coefficient in \eqref{GenericStochasticSystem} serves as the weak derivative of $\bm{z}(t)$.
\end{lemma}

\begin{proof}
    By definition, the weak derivative of a generic function $\bm{h}(t): \mathbb{R}_{\geq 0} \to \mathbb{R}^{n_x}$, where $\bm{h}(t) \in \mathcal{L}^{\mathcal{F}_t}_1[0,\infty)$, is a function $\bm{a}(t): \mathbb{R}_{\geq 0} \to \mathbb{R}^{n_x}$, with $\bm{a}(t) \in \mathcal{L}^{\mathcal{F}_t}_1[0,\infty)$, that satisfies the following condition \citep{knabner2003numerical}:
   % \vspace{-1mm}
\begin{align}
\label{WeakDerivative}
    \lim_{t \to \infty}\int_{0}^{t} \bm{h}'(s)\dot{\bm{\psi}}(s)ds &= - \lim_{t \to \infty}\int_{0}^{t} \bm{a}'(s) \bm{\psi}(s) ds,
\end{align}
for any $\mathcal{F}_t$-measurable test function $\bm{\psi}(t)$, such that $\bm{\psi}(0) = \displaystyle\lim_{t \to \infty}\bm{\psi}(t) = \bm{0}$.

%\vspace{-2mm}
By applying the chain rule, assuming the time derivative of $\bm{h}(t)$ exists in the domain $t \in \mathbb{R}_{\geq 0}$, we obtain\footnote{The shorthand $\int_{0}^{\infty} \bm{z}(t)dt$ denotes $\displaystyle\lim_{t\to \infty}\int_{0}^{t} \bm{z}(s)ds$.}
%\vspace{-1mm}
\begin{align}
     \left.\bm{h}'(t)\bm{\psi}(t) \right|^\infty_0 &= \int^\infty_0 \left(\dot{\bm{h}}'(t)\psi(t) {+} \bm{h}'(t)\dot{\bm{\psi}}(t)\right)dt, \nonumber\\
     \int^\infty_0\bm{h}'(t)\dot{\bm{\psi}}(t)dt &= -\int^\infty_0 \dot{\bm{h}}'(t)\psi(t)dt,  \label{ExpandedWeakDerivative}
\end{align}
where we used the fact that $\left.\bm{h}'(t)\bm{\psi}(t) \right|^\infty_0 = 0$.  Therefore, it follows that the weak derivative of $\bm{h}(t)$, denoted by $\partial \bm{h}(t)$, is equivalent to its classical derivative $\dot{\bm{h}}(t)$, provided the latter exists. This result will be applied in the subsequent proof.
%\vspace{-1mm}

To proceed, we apply the Itô formula to the stochastic system given by \eqref{GenericStochasticSystem}, resulting in
%\vspace{-1mm}
\begin{align}
\label{ItoFormula}
 \bm{x}(t)  =   \bm{x}(0) + \int_0^t \bm{f}(s)ds + \int_0^t\bm{g}(s)dm(s), 
\end{align}
where, for simplicity, we assume $\bm{f}(t) = \bm{f}(\bm{x}, \bm{u}, \bm{w}, t)$ and $\bm{g}(t) = \bm{g}(\bm{x}, \bm{u}, \bm{w}, t)$.
Consequently, the weak derivative of $\bm{z}(t)$ is given by
%\vspace{-4mm}
\begin{align}
    \partial\bm{z}(t) &{=} \partial\Expect{\bm{x}(t)} {=} \partial\mathcal{E}\Big\{\overbrace{\bm{x}(0)}^{\text{1st}} {+} \overbrace{\int_0^t \!\!\!\bm{f}(s)ds}^{\text{2nd}} {+} \overbrace{\int_0^t\!\!\!\bm{g}(s)dm(s)}^{\text{3rd}}\!\Big\}.  \label{TermsExpand}
\end{align}
Since both the expected value and the weak derivative are linear operators, we can expand the three terms in \eqref{TermsExpand} separately.
%\vspace{-1mm}

The first term in \eqref{TermsExpand} is expanded as 
\begin{align}
\label{Term1}
    \partial\Expect{\bm{x}(0)} = \partial\bm{x}_0 = 0,
\end{align}
which was computed considering the fact that the classical derivative of $\bm{x}_0$ exists and is equal to zero. 
%\vspace{-1mm}

Given that function $\bm{f}(t)$ is $ \mathcal{F}_t$-measurable, the second term in \eqref{TermsExpand} is expanded as 
\begin{align}
    \partial\Expect{\!\int_0^t \!\bm{f}(s)ds } &{=} \partial\Expect{\!\int_0^t \bm{f}(s)ds \Big| \mathcal{F}_t } {=} \partial\! \int_0^t \bm{f}(s)ds {=} \bm{f}(t), \label{Term2}
\end{align}
where we used the fact that the classical derivative of $\int_0^t \bm{f}(s)ds$ exists and is equal to $\bm{f}(t)$.
%\vspace{-1mm}

Finally, to expand the third term in \eqref{TermsExpand}, let $\Pi_n = \{t_0,t_1,\cdots,t_n\}$ be a partition of $[0,~t]$, for $n \in \mathbb{N}$, that is, $0 = t_0 < t_1 < \cdots < t_n = t$. Then, we can replace $\bm{g}(t)$ with a simple process $\bm{g}_n(t)$ to compute the Itô's integral~\citep{shreve2004stochastic}, such that $\displaystyle\lim_{n \to \infty} \Expect{\int_0^\infty\big|\bm{g}(t) - \bm{g}_n(t)\big|dt} = 0$. Therefore, given the assumption that $\bm{g}(t)$ is $ \mathcal{F}_t$-measurable, we have that
\begin{align}
\label{Term3}
   &\partial\Expect{\int_0^t\bm{g}(s)dm(s)\Big|\mathcal{F}_{t}}  \\
    &= \partial\mathcal{E} \Big\{\lim_{n \to \infty}\sum_{k = 0}^{n-1} \bm{g}_n(t_k) \left[m(t_{k+1}) - m(t_k)\right]\Big|\mathcal{F}_{t_k}\Big\}, \nonumber \\
    &= \partial\Big(\lim_{n \to \infty}\sum_{k = 0}^{n-1} \bm{g}_n(t_k) \mathcal{E} \Big\{\left[m(t_{k+1}) {-} m(t_k)\right]\Big|\mathcal{F}_{t_k} \Big\}\Big) {=} 0, \nonumber
\end{align}
where it was used the facts that $\bm{g}(t) \in \mathcal{L}^{\mathcal{F}_t}_1[0,\infty)$, $m(0) = 0$, and $\mathcal{E} \Big\{\left[m(t_{k+1}) - m(t_k)\right]\Big\} = 0$. 

\vspace{-2mm}
Accordingly, combining \eqref{Term1}, \eqref{Term2}, and \eqref{Term3}, we conclude that
%\vspace{-2mm}
\begin{align}
    \partial\bm{z}(t) = \bm{f}(t),
\end{align}
%\vspace{-2mm}
which completes the proof. $\qquad\qquad\qquad\qquad\quad\qquad\;$\qed%\begin{flushright}\qed\end{flushright}
\end{proof}

\vspace{-1mm}
The stochastic {robust} $\mathcal{W}_\infty$ {controller} can be designed using Lemmas \ref{TheoremSobolevspace} and \ref{TheoremWeakDerivative}. To achieve this, the OCP \eqref{WinfControlProblem2} can be solved using dynamic programming through the Hamilton-Jacobi-Bellman-Isaacs (HJBI) equation \citep{yong2012stochastic}, which results in the following formulation:
\begin{align}
    v_t + \min_{\bm{u} \in \mathcal{U}} \max_{\bm{w} \in \mathcal{W}} \Expect{\mathcal{H}(v_x,v_{xx},\bm{x},\bm{u},\bm{w},t)} = 0,
    \label{HJBIeq}
\end{align}
with the Hamiltonian
%\begin{align}
%\label{Hamiltonian1}
		$\mathcal{H}(v_x,v_{xx}, \bm{x},\bm{u},\bm{w},t) = v'_x d\bm{x} + \dfrac{1}{2}d\bm{x}' v_{xx}d\bm{x} 
  + \left(\bm{x}' \bm{\Gamma}_0\bm{x} + \bm{f}' \bm{\Gamma}_1\bm{f} - \gamma_{\mathcal{W}}^2\bm{w}' \bm{w}\right)dt$, and boundary condition $v(\bm{0},t) = 0$, $\forall t \in \mathbb{R}_{\geq 0}$. The OCP thus leads to finding a solution $v(\bm{x},t) > 0$ of the {resulting} second-order partial differential equation {from} \eqref{HJBIeq}.

\begin{lemma}
\label{TheoremWinf}
If there exists a solution $v(\bm{x},t) > 0$ to the {resulting} second-order partial differential equation {from} \eqref{HJBIeq}, then the closed-loop system \eqref{GenericStochasticSystem}, under the control law $\bm{u}^* = \displaystyle {\arg}\min_{\bm{u} \in \mathcal{U}} \max_{\bm{w} \in \mathcal{W}} \Expect{\mathcal{H}(v_x,v_{xx},\bm{x},\bm{u},\bm{w},t)}$ satisfies inequality \eqref{StochasticWinfController}, with $c_3 = \Expect{v(\bm{x}(0), 0) - \lim_{s \to \infty} v(\bm{x}(s), s)}$. Furthermore, solving the {HJBI equation} \eqref{HJBIeq} {leads to the solution of the optimal control law} \eqref{WinfControlProblem2}.
\end{lemma}

\begin{proof}
%{COMEÇOU COM ESSA EQUAÇÃO SOLTA}
%\begin{align}
%    \Expect{dv} {=} -\bm{x}'\bm{\Gamma}_0\bm{x} - (\bm{f}^*)'\bm{\Gamma}_1\bm{f}^* + \gamma_{\mathcal{W}}^2(\bm{w}^*)'\bm{w}^*
%\end{align}
    By considering $\bm{z}^*$, $\bm{f}^*$, and $\bm{w}^*$ as the respectively min-max optimized values of $\bm{z}(t)$, $\bm{f}(t)$, and $\bm{w}(t)$, the proof can be carried out as follows.
    
     From the HJBI equation \eqref{HJBIeq}, assuming that the solution $v(\bm{x}, t) > 0$ exists, one can obtain
     \vspace{-1mm}
    \begin{align}
    \label{LyapunovDeferential2}
        \Expect{dv} {=} \left(-\bm{x}'\bm{\Gamma}_0\bm{x} - (\bm{f}^*)'\bm{\Gamma}_1\bm{f}^* + \gamma_{\mathcal{W}}^2(\bm{w}^*)'\bm{w}^*\right)dt.
    \end{align}
    Integrating both sides of \eqref{LyapunovDeferential2}, in the time interval $t \in \mathbb{R}_{\geq 0}$, we achieve
    \vspace{-1mm}
        \begin{align}
        -c_3 &=  -||\bm{z}^*||^2_{\mathcal{W}_{1,2,\bm{\Gamma}}} + \gamma_{\mathcal{W}}^2 ||\bm{w}^*||^2_{\mathcal{L}_2}, \nonumber\\
        c_3 &=  \displaystyle \min_{\bm{u} \in \mathcal{U}} \max_{\bm{w} \in \mathcal{W}} \big( ||\bm{z}(t)||^2_{\mathcal{W}_{1,2,\bm{\Gamma}}} - \gamma_{\mathcal{W}}^2 ||\bm{w}(t)||^2_{\mathcal{L}_2} \big). \label{LyapunovDeferential3}
    \end{align}
    Therefore, solving the {HJBI equation} \eqref{HJBIeq} {leads to the solution of the optimal control law} \eqref{WinfControlProblem2}{, which is then computed as}%. The optimal control law is then computed by 
    \begin{align}
        \bm{u}^* &= \displaystyle {\arg}\min_{\bm{u} \in \mathcal{U}} \max_{\bm{w} \in \mathcal{W}} \Expect{\mathcal{H}(v_x,v_{xx},\bm{x},\bm{u},\bm{w},t)}, \nonumber\\ 
        &= \displaystyle \arg\min_{\bm{u} \in \mathcal{U}} \max_{\bm{w} \in \mathcal{W}} \big( ||\bm{z}(t)||^2_{\mathcal{W}_{1,2,\bm{\Gamma}}} - \gamma_{\mathcal{W}}^2 ||\bm{w}(t)||^2_{\mathcal{L}_2} \big). \label{OptimalControl}
    \end{align} 

\vspace{-2mm}
    From the results of Lemma \ref{TheoremSobolevspace}, we can conclude that the closed-loop system \eqref{GenericStochasticSystem}, under the control law \eqref{OptimalControl}, satisfies inequality \eqref{StochasticWinfController}. $\qquad\qquad\qquad\qquad\quad\qquad\qquad\qquad\qquad$\qed%\begin{flushright}\qed\end{flushright}
    
\end{proof}

\vspace{-2mm}
\section{LMI-Based Synthesis of Dynamic Output-Feedback $\mathcal{W}_\infty$ Controllers}
\vspace{-2mm}

%The following theorem provides LMI conditions for synthesizing a linear $\mathcal{W}_\infty$ controller for the stochastic system~\eqref{GenericStochasticSystem}, ensuring mean-square ultimate boundedness. This result constitutes the main contribution of the manuscript.

As the main contribution, the following theorem provides LMI conditions to synthesize a linear $\mathcal{W}_\infty$ controller that ensures mean-square ultimate boundedness for the stochastic system~\eqref{GenericStochasticSystem}.

\begin{theorem}
Consider the stochastic system \eqref{GenericStochasticSystem} in closed-loop with the dynamic-output feedback control law
%\vspace{-1mm}
\begin{align}
\label{u1} 
\mathcal{K}_1: \begin{cases}
    \bm{u} &= \bm{C}_f\bm{x}_f + \bm{D}_f\bm{y}, \\ 
    d\bm{x}_f &= \left(\bm{A}_f\bm{x}_f + \bm{B}_f\bm{y}\right)dt.
\end{cases}
\end{align}
Assume that there exist matrices 
\small$\bm{S}$, $\bm{W}$, $\bm{B}_s$, $\bm{C}_s$, $\bm{D}_f$\normalsize, and \small$\bm{N}$ \normalsize
satisfying the SDP
\begin{gather}
	\label{OptLMIOutputFeedback}
    \min\limits_{\bm{S},\bm{W}, \bm{B}_s, \bm{C}_s, \bm{D}_f} \text{tr}\left\{\bm{N}\right\}, \\ \nonumber
	s.t.: \begin{cases} 
    \begin{bmatrix}
			\bm{N} & \bm{E}'_x \\
			\bm{E}_x & \bm{W}\end{bmatrix} > 0, ~~		\begin{bmatrix}
			\bm{S} & \IdentityMatrix{n_x} \\
			\IdentityMatrix{n_x} & \bm{W}'\end{bmatrix} > 0, \\
		\begin{bmatrix}
			\bm{\Psi}_1 & \bm{*} & \bm{*}\\
			\ZerosMatrix{n_x}{n_x} & \bm{W}\bm{\Psi}_3\bm{W}' & \bm{*} \\
			\bm{\Psi}_4 & \bm{\Psi}_5\bm{W}' & \bm{\Psi}_6
		\end{bmatrix} < 0,
	\end{cases}
\end{gather} 
where $\gamma_{\mathcal{W}} = \sqrt{\gamma^*}$ is a given $\mathcal{W}_\infty$-index, and
\begin{align}
    \bm{\Psi}_1& \triangleq \bm{A}\bm{S} + \bm{B}\bm{C}_s + \left(\bm{A}\bm{S} + \bm{B}\bm{C}_s\right)', \\
    \bm{W}\bm{\Psi}_3\bm{W}' &\triangleq \bm{A}'\bm{W}' + \bm{A}_y'\bm{B}'_s + \bm{W}\bm{A} + \bm{B}_s\bm{A}_y,  \\
\bm{\Psi}'_4 &\triangleq \begin{bmatrix}
	\bm{D} & \bm{S} & (\bm{A}\bm{S} + \bm{B}\bm{C}_s)'
\end{bmatrix}, \\
\bm{W}\bm{\Psi}'_5 &\triangleq \begin{bmatrix}
	\bm{W}{\bm{D}} & \IdentityMatrix{n_x} & \big(\bm{A} + \bm{B}\bm{D}_f\bm{A}_y\big)'
\end{bmatrix},  \\
\bm{\Psi}_6 &\triangleq \begin{bmatrix}
	-\gamma^*\IdentityMatrix{1} & \ZerosMatrix{n_w}{n_x} & \bm{D}' \\
	\ZerosMatrix{n_x}{n_w} & -\bm{\Gamma}^{-1}_0 & \ZerosMatrix{n_x}{n_x} \\
	\bm{D} & \ZerosMatrix{n_x}{n_x} & -\bm{\Gamma}^{-1}_1
\end{bmatrix},
\end{align}
\vspace{-1mm}
such that
\begin{align*}
	\bm{A}_f &{=} \bm{T}(\bm{S}{-}\bm{W}^{-1})^{{-}1}\big( \bm{A}\bm{S} {+} \bm{B}\bm{C}_s {-} \left( \bm{B}\bm{D}_f {-} \bm{W}^{-1}\bm{B}_s \right)\bm{C}\bm{S} \\
	&{+} (\bm{W}^{-1})'(\bm{A}+\bm{B}\bm{D}_f\bm{C})'\big)\bm{T}^{-1} \\
	\bm{B}_f &= \bm{T}(\bm{S}-\bm{W}^{-1})^{-1}\left( \bm{B}\bm{D}_f - \bm{W}^{-1}\bm{B}_s\right), \\
	\bm{C}_f &= \left( \bm{C}_s - \bm{D}_f\bm{C}\bm{S} \right)\bm{T}^{-1}, 
\end{align*}
where $\bm{T}$ is a symmetric and positive definite matrix that can be arbitrarily chosen. 

\vspace{-2mm}
Then the following properties hold:
\begin{enumerate}
\vspace{-1mm}
    \item[1)] The closed-loop {system} trajectories satisfy the stochastic dissipation inequality
    \begin{align}
    \label{dissineq}
    \Expect{dv} \le \left(-\alpha \|\bm{\mathcal{X}}\|^2 
    + \beta \|\bm{w}\|^2 
    + \mathrm{tr}(\bm{N}) \right)dt,
    \end{align}
    for some $\alpha,\beta \in \mathbb{R}_{>0}$, with $\bm{\mathcal{X}} = \begin{bmatrix}
        \bm{x} & \bm{x}_f
    \end{bmatrix}'$.
    \item[2)] There exist constants $c_1,c_2 \in \mathbb{R}_{>0}$ such that the closed-loop system is \emph{mean-square ultimately bounded}, with  
    \[
    \limsup_{t\to\infty}
    \Expect{\|\bm{\mathcal{X}}(t)\|_2^2}
    \le 
    \dfrac{c_1 \|\bm{w}\|^2_{\mathcal{L}_\infty} + \mathrm{tr}(\bm{N})}{c_2}.
    \]
\end{enumerate}
\end{theorem}

\begin{proof}
The closed-loop system obtained by applying the dynamic output-feedback controller \eqref{u1} to \eqref{GenericStochasticSystem}, denoted by $\mathcal{P}_2(\mathcal{K}_1)$, can be written as
\begin{align}
\nonumber d\bm{\mathcal{X}} &{=} \Bigg(\underbrace{\begin{bmatrix}
    \bm{A}+\bm{B}\bm{D}_f\bm{A}_y & \bm{B}\bm{C}_f \\
    \bm{B}_f\bm{A}_y & \bm{A}_f
\end{bmatrix}}_{\bm{A}_c} \bm{\mathcal{X}} {+} \underbrace{\begin{bmatrix}
    \bm{D} \\
    \bm{0}
\end{bmatrix}}_{\bm{D}_c}\bm{w}\Bigg)dt {+} \underbrace{\begin{bmatrix}
    \bm{E}_x \\
    \bm{0}
\end{bmatrix}}_{\bm{E}_c}dm, \\
    &=  \left(\bm{A}_c\bm{\mathcal{X}} {+} \bm{D}_c\bm{w}\right)dt {+} \bm{E}_c dm. \label{Closedloop}
\end{align}

\vspace{-2mm}
To enforce the $\mathcal{W}_\infty$ performance criterion, we invoke Lemma~\ref{TheoremWinf} and write the corresponding Hamiltonian inequality
\vspace{-1mm}
\begin{align}
\label{Hamiltonian2}
		\Expect{dv} + \left(\bm{x}'\bm{\Gamma}_0\bm{x} + \bm{f}'\bm{\Gamma}_1\bm{f} {-} \gamma_{\mathcal{W}}^2\bm{w}'\bm{w}\right)dt \leq 0,
\end{align}
with the quadratic Lyapunov candidate
\vspace{-1mm}
\begin{align}
    v &= \frac{1}{2}\bm{\mathcal{X}}'\bm{P}\bm{\mathcal{X}} > 0.  \label{Lyap} 
\end{align}
Since $v$ has no explicit time dependence, we have $v_t = 0$, $v_{\bm{\mathcal{X}}} = \bm{P}\bm{\mathcal{X}}$, $v_{\bm{\mathcal{X}}\bm{\mathcal{X}}} =  \dfrac{1}{2}\bm{P}.$
Using the stochastic calculus identities $dtdt = dtdm = dmdt = 0$ and $dmdm = dt$ \citep{shreve2004stochastic}, it follows that
\vspace{-1mm}
\begin{align}
\Expect{dv} &= v'_{\bm{\mathcal{X}}}d\bm{\mathcal{X}} {+} \dfrac{1}{2}d\bm{\mathcal{X}}' v_{\bm{\mathcal{X}}\bm{\mathcal{X}}}d\bm{\mathcal{X}}, \\
    &= \bm{\mathcal{X}}'\bm{P}\left(\bm{A}_c\bm{\mathcal{X}} {+} \bm{D}_c\bm{w}\right)dt + \text{tr}\left\{\bm{E}'_c\bm{P}\bm{E}_c\right\}dt. \label{dvlmi1}
\end{align}

\vspace{-2mm}
The second term \GR{in} \eqref{dvlmi1} corresponds to the mean ultimate bound induced by $\bm{E}_c$ \GR{and is thus minimized}. \GR{Specifically, the cost functional minimizes $\mathrm{tr}(\bm{N})$} as an upper bound for $\mathrm{tr}(\bm{E}_c' \bm{P} \bm{E}_c)$.

Therefore, combining \eqref{Closedloop}, \eqref{Hamiltonian2}, and \eqref{Lyap}, the control design problem becomes
\begin{align}
 &\min\limits_{\gamma, \bm{P}, \bm{A}_f, \bm{B}_f, \bm{C}_f, \bm{D}_f} \text{tr}\left\{\bm{N}\right\},  \label{Opc1} \\
 s.t.{:}& \begin{cases} 
 \bm{N} - \bm{E}'_c\bm{P}\bm{E}_c > 0, ~~ \bm{P} > 0 \\
\bm{\mathcal{X}}'\bm{P}\left(\bm{A}_c\bm{\mathcal{X}} {+} \bm{D}_c\bm{w}\right) {+} \Big(\bm{x}'\bm{\Gamma}_0\bm{x} {+} \bm{f}'\bm{\Gamma}_1\bm{f}  - \gamma^*\bm{w}'\bm{w}\Big) {<} 0, 
    \end{cases} \nonumber
\end{align}
where $\bm{f} = \bm{U} \bm{\mathcal{X}} + \bm{D}\bm{w}$, with $\bm{U} \triangleq \begin{bmatrix} \bm{A} + \bm{B} \bm{D}_f \bm{A}_y & \bm{B} \bm{C}_f \end{bmatrix}$.

\vspace{-1mm}
To express the constraints in \eqref{Opc1} as an SDP, we proceed as follows.
First, group the terms in the last inequality of \eqref{Opc1} to obtain
\begin{gather}
	\label{LMI21}
    \begin{bmatrix}
        \bm{\mathcal{X}} \\
        \bm{w}
    \end{bmatrix}'\!\!\begin{bmatrix}
		\bm{A}'_c\bm{P} {+} \bm{P}\bm{A}_c {+} \bm{F} {+} \bm{U}'\bm{\Gamma}_1\bm{U} & \bm{*} \\
		\bm{D}'_c\bm{P} {+} \bm{D}'\bm{\Gamma}_1\bm{U} & \bm{D}'\bm{\Gamma}_1\bm{D}{-}\gamma^*\IdentityMatrix{1}
	\end{bmatrix}\!\!\begin{bmatrix}
        \bm{\mathcal{X}} \\
        \bm{w}
    \end{bmatrix} < 0,
\end{gather}\normalsize
where $\bm{F} \triangleq \begin{bmatrix}
	\bm{\Gamma}_0 & \ZerosMatrix{n_x}{n_x} \\ \ZerosMatrix{n_x}{n_x} & \ZerosMatrix{n_x}{n_x}
\end{bmatrix}$\normalsize.

Next, apply the similarity transformation \citep{chen1984linear} $\mathcal{X} = \bm{P}^{-1}\bm{M}'\mathcal{Y}$, where
\vspace{-2mm}
\begin{gather}
\label{MatrixP}
    \bm{P}^{-1} \triangleq  \begin{bmatrix}
	\bm{S}  & \bm{T} \\ \bm{T}' & \bm{R}
\end{bmatrix}
\end{gather}
and $\bm{M} \triangleq \begin{bmatrix}
	\IdentityMatrix{n_x} & \ZerosMatrix{n_x}{n_x}\\
	\IdentityMatrix{n_x} & -\bm{T}\bm{R}^{-1}
\end{bmatrix}$\normalsize,
which transforms to the second inequality in \eqref{Opc1} into
\vspace{-1mm}
\normalsize\begin{gather}
	\label{LMI23}
	\mathcal{Y}'_a \bm{M}\bm{P}^{-1}\bm{M}'\mathcal{Y}_a = \mathcal{Y}'_a\begin{bmatrix}
		\bm{S} & \bm{S} - \bm{Q} \\
		\bm{S} - \bm{Q} & \bm{S} - \bm{Q}
	\end{bmatrix}\mathcal{Y}_a > 0.
\end{gather}
Applying the same transformation to inequality \eqref{LMI21} yields
\normalsize\begin{gather}
	\label{LMI24}
    \begin{bmatrix}
		\!\bm{M}\bm{P}^{-1}\!\!\left(\bm{A}'_a\bm{P} {+} \bm{P}\bm{A}_a {+} \bm{F} {+} \bm{U}'\bm{\Gamma}_1\bm{U}\right)\!\!\bm{P}^{-1}\bm{M}' & \!\!\!\!\!\!\!\!\bm{*} \\
		\left(\bm{D}'_a\bm{P} {+} \bm{D}'\bm{\Gamma}_1\bm{U}\right)\bm{P}^{-1}\bm{M}' & \!\!\!\!\!\!\!\!\bm{D}'\bm{\Gamma}_1\bm{D}{-}\gamma^*\IdentityMatrix{1}
	\end{bmatrix} \!\!{<} 0,
\end{gather}\normalsize
where $\bm{Q} \triangleq \bm{T}\bm{R}^{-1}\bm{T}'$.

\vspace{-1mm}
Applying the Schur complement to \eqref{LMI24} and defining $\bm{L} \triangleq \begin{bmatrix} \bm{S}' & \bm{S}' - \bm{Q} \end{bmatrix}'$, we obtain
\normalsize\begin{gather}
	\label{LMI25}
\begin{bmatrix}
		\bm{M}\bm{P}^{-1}\left(\bm{A}'_a\bm{P} {+} \bm{P}\bm{A}_a\right)\bm{P}^{-1}\bm{M}' & \!\!\bm{*} & \bm{*} & \bm{*}\\
		\bm{D}'_a\bm{M}' & \!\!{-}\gamma^*\IdentityMatrix{1} & \bm{*} & \bm{*} \\
		\bm{L}' & \!\!\bm{0} & -\bm{\Gamma}^{{-}1}_0 & \bm{*} \\
		\bm{U}\bm{P}^{-1}\bm{M}' & \!\!\bm{D} & \bm{0} & -\bm{\Gamma}^{{-}1}_1
	\end{bmatrix} < 0.
\end{gather}\normalsize

\vspace{-1mm}
Then, \eqref{LMI25} is rewritten in the compact block form
\begin{gather}
	\label{LMI26}
	\begin{bmatrix}
		\bm{\Psi}_1 & \bm{*} & \bm{*}\\
		\bm{\Psi}_2 - \bm{A}_m & \bm{\Psi}_3 & \bm{*} \\
		\bm{\Psi}_4 & \bm{\Psi}_5 & \bm{\Psi}_6
	\end{bmatrix} < 0,
\end{gather} 
where
\small$\begin{bmatrix}
	\bm{\Psi}_1 & \bm{*} \\
	\bm{\Psi}_2 - \bm{A}_m & \bm{\Psi}_3
\end{bmatrix} \triangleq \bm{M}\bm{P}^{{-}1}\left(\bm{A}'_a\bm{P} + \bm{P}\bm{A}_a\right)\bm{P}^{{-}1}\bm{M}'$\normalsize, with \small$\bm{\Psi}_1 \triangleq \bm{A}\bm{S} + \bm{B}(\bm{D}_f \bm{C}\bm{S} + \bm{C}_f \bm{T}') + \left(\bm{A}\bm{S} + \bm{B}(\bm{D}_f\bm{C}\bm{S} + \bm{C}_f\bm{T})\right)'$\normalsize, \small$\bm{\Psi}_2\triangleq (\bm{S}-\bm{Q})'(\bm{A}+\bm{B}\bm{D}_f \bm{C})' + (\bm{A}+\bm{B}\bm{D}_f \bm{C})\bm{S}'+\bm{B}\bm{C}_f\bm{T}' - \bm{T}\bm{R}^{-1}\bm{B}_f\bm{C}\bm{S}'$\normalsize, \small$\bm{A}_m \triangleq \bm{T}\bm{R}^{-1}\bm{A}_f\bm{T}'$\normalsize, \small$\bm{\Psi}_3 = (\bm{S} - \bm{Q})\bm{A}' + \bm{A}(\bm{S} - \bm{Q})' + (\bm{S} - \bm{Q})\bm{C}'(\bm{D}_f\bm{B}' - \bm{B}'_f\bm{R}^{-1}\bm{T}') + (\bm{B}\bm{D}_f - \bm{T}\bm{R}^{-1}\bm{B}_f)\bm{C}(\bm{S} - \bm{Q})'$\normalsize,
\small$\begin{bmatrix}
	\bm{\Psi}_4 &
	\bm{\Psi}_5
\end{bmatrix} \triangleq \begin{bmatrix}
	\bm{D}'_a\bm{M}' \\ \bm{L}' \\ \bm{U}\bm{P}^{-1}\bm{M}'
\end{bmatrix}$\normalsize, 
\small$\bm{\Psi}_6 \triangleq \begin{bmatrix}
	-\gamma^2\IdentityMatrix{1} & \bm{*} & \bm{*} \\
	\ZerosMatrix{n_x}{n_w} & -\bm{\Gamma}^{{-}1}_0 & \bm{*} \\
	\bm{D} & \ZerosMatrix{n_x}{n_x} & -\bm{\Gamma}^{{-}1}_1
\end{bmatrix}$\normalsize.

Setting $\bm{A}_m = \bm{\Psi}_2$, introducing the change of variables $\bm{C}_s \triangleq \bm{D}_f\bm{C}\bm{S} + \bm{C}_f\bm{T}$, $\bm{B}_s \triangleq \bm{W}(\bm{B}\bm{D}_f-\bm{T}'\bm{R}^{-1}\bm{B}_f)$, with $\bm{W} \triangleq \left(\bm{S} - \bm{Q} \right)^{-1}$, and applying the following congruent transformations to the inequalities \eqref{LMI23} and \eqref{LMI26},
\begin{gather}
	\begin{bmatrix}
		\IdentityMatrix{0}  & \ZerosMatrix{1}{1} \\ \ZerosMatrix{1}{1} & \bm{W}
	\end{bmatrix}
	\begin{bmatrix}
		\bm{S} & \bm{S} - \bm{Q} \\
		\bm{S} - \bm{Q} & \bm{S} - \bm{Q}
	\end{bmatrix}
	\begin{bmatrix}
		\IdentityMatrix{0}  & \ZerosMatrix{1}{1} \\ \ZerosMatrix{1}{1} & \bm{W}'
	\end{bmatrix} > 0, \\
	\begin{bmatrix}
	    \IdentityMatrix{0}  & \ZerosMatrix{1}{1}  & \ZerosMatrix{1}{1} \\
		\ZerosMatrix{1}{1}  & \bm{W} & \ZerosMatrix{1}{1} \\
		\ZerosMatrix{1}{1}  & \ZerosMatrix{1}{1} & \IdentityMatrix{0}
	\end{bmatrix}\begin{bmatrix}
		\bm{\Psi}_1 & \bm{*} & \bm{*}\\
		\ZerosMatrix{1}{1} & \bm{\Psi}_3 & \bm{*} \\
		\bm{\Psi}_4 & \bm{\Psi}_5 & \bm{\Psi}_6
	\end{bmatrix}\begin{bmatrix}
	\IdentityMatrix{0}  & \ZerosMatrix{1}{1}  & \ZerosMatrix{1}{1} \\ \ZerosMatrix{1}{1} & \bm{W}' & \ZerosMatrix{1}{1} \\
	\ZerosMatrix{1}{1} & \ZerosMatrix{1}{1} & \IdentityMatrix{0}
\end{bmatrix} < 0,
\end{gather} 
yield the second and third inequalities in \eqref{OptLMIOutputFeedback}, which are linear in $\bm{S}$, $\bm{W}$, $\bm{B}_s$, $\bm{C}_s$, and $\bm{D}_f$.

To derive the first inequality in \eqref{OptLMIOutputFeedback}, consider the LDU factorization of \eqref{MatrixP}:
\begin{gather}
    \bm{P}^{-1} {=} \underbrace{\begin{bmatrix}
        \IdentityMatrix{} & \bm{T}\bm{R}^{-1} \\
        \ZerosMatrix{}{} & \IdentityMatrix{}
    \end{bmatrix}}_{\bm{L}}\! \underbrace{\begin{bmatrix}
        \bm{S}-\bm{T}\bm{R}^{-1}\bm{T} & \ZerosMatrix{}{} \\
        \ZerosMatrix{}{} & \bm{R}
    \end{bmatrix}}_{\bm{D}} \!\underbrace{\begin{bmatrix}
        \IdentityMatrix{} &  \ZerosMatrix{}{}  \\
       -\bm{R}^{-1}\bm{T}^{T}& \IdentityMatrix{}
    \end{bmatrix}}_{\bm{U}}.
\end{gather}
From this factorization, we obtain\vspace{-1mm}
\small\begin{align*}
    \bm{P} &{=} \bm{U}^{-1}\bm{L}^{-1}\bm{D}^{-1} {=} \begin{bmatrix}
\bm{W} ^{-1} & \!\!\!\!-\bm{W}^{-1} \bm{T} \bm{R}^{-1} \\
- \bm{R}^{-1} \bm{T}' \bm{W}^{-1} &~~ \!\!\!\!\bm{R}^{-1} \! +\! \bm{R}^{-1} \bm{T}' \bm{W}^{-1} \bm{T} \bm{R}^{-1}
\end{bmatrix},\vspace{-2mm}
\end{align*}\normalsize
where  $\bm{W}^{-1} = \big(\bm{S} - \bm{T} \bm{R}^{-1} \bm{T}'\big)^{-1} = \left(\bm{S} - \bm{Q}\right)^{-1}$. 
The second term in the first inequality of \eqref{Opc1} can be expanded,
%\begin{gather}
    $\bm{E}'_c\bm{P}\bm{E}_c = \begin{bmatrix}
    \bm{E}_x \\
    \bm{0}
\end{bmatrix}'\bm{P}\begin{bmatrix}
    \bm{E}_x \\
    \bm{0}
\end{bmatrix} = \bm{E}'_x \bm{W}^{-1} \bm{E}_x$,
%\end{gather}
Thus, the first constraint in \eqref{Opc1} becomes
%\begin{gather}
%\label{IneqTrace1}
 $\bm{N} - \bm{E}'_x \bm{W}^{-1} \bm{E}_x > 0$,
%\end{gather}
which, by applying the Schur complement, leads directly to the first inequality in \eqref{OptLMIOutputFeedback}. 

To complete the proof, we must show that feasibility of 
\eqref{OptLMIOutputFeedback} guarantees \emph{stochastic dissipation} 
and \emph{mean-square ultimate boundedness} of the closed-loop system.

\noindent1) Stochastic dissipation:
Using \eqref{dvlmi1} and the first inequality in \eqref{Opc1}, we obtain
   $\mathrm{tr}(\bm{E}_c' P \bm{E}_c)
   \le \mathrm{tr}(\bm{N})$.
Substituting this bound into the Hamiltonian inequality 
\eqref{Hamiltonian2} yields
\begin{align}
\Expect{dv} 
{\leq} \big( 
{-}\bm{x}' \Gamma_0 \bm{x}
{-} \bm{f}' \Gamma_1 \bm{f}
{+} \gamma_{\mathcal{W}}^2 \bm{w}' \bm{w}
{+} \mathrm{tr}(\bm{N})\,\big)dt .
\end{align}
This is a stochastic dissipation inequality of the form
\begin{gather}
\label{DissipationInequality}
\Expect{dv} \le -\alpha \Expect{\|\bm{\mathcal{X}}\|_2^2} dt 
+ \beta \|\bm{w}\|_2^2 dt 
+ \mathrm{tr}(\bm{N}) dt,
\end{gather}
for some $\alpha,\beta \in \mathbb{R}_{\geq 0}$ implied by the feasibility of 
\eqref{OptLMIOutputFeedback}.

\noindent 2) Mean-square ultimate boundedness:
Integrating \eqref{DissipationInequality} from $0$ to $t$ gives
\begin{align}\nonumber
\Expect{v(\mathcal{X}(t))} &- \Expect{v(\mathcal{X}(0))}
\le 
-\alpha \int_0^t \Expect{\|\mathcal{X}(s)\|_2^2}ds
\\ \label{eq:preBound}
&\quad
+ \beta\int_0^t \|w(s)\|_2^2\,ds
+ t\,\mathrm{tr}(N).
\end{align}
Because $v(\mathcal{X}) = \mathcal{X}' P\mathcal{X}$ and $P\succ 0$, we have
%\begin{gather}
$\Expect{v(\mathcal{X}(t))} \geq \lambda_{\min}(P)\,\Expect{\|\mathcal{X}(t)\|_2^2},
\qquad
\forall \mathcal{X}\in\mathbb{R}^n$. %\label{IneqLyap1}
%\end{gather}
%\eqref{IneqLyap1}
Using that equation and dividing both sides of \eqref{eq:preBound} by $t$, it leads
\begin{align}
\label{eq:timeAvg2}
\!\!\!\alpha\,
\frac{1}{t}\!\!\int_0^t \!\!\!\Expect{\|\mathcal{X}(s)\|_2^2}\,ds
&{\le}
\dfrac{v(\mathcal{X}(0))}{t} {+} \beta\,\frac{1}{t}\int_0^t \!\!\!\|w(s)\|_2^2\,ds
{+} \mathrm{tr}(N). 
\end{align}
Taking $\limsup_{t\to\infty}$ to both sides of \eqref{eq:timeAvg2}, we obtain
$\limsup_{t\to\infty} 
\frac{1}{t}\int_0^t \|w(s)\|_2^2\,ds
\le \|w\|_{\mathcal{L}_\infty}^2,$
and using that 
$\limsup_{t\to\infty} v(\mathcal{X}(0))/t = 0$ {yields}
$\alpha\,
\limsup_{t\to\infty}
\frac{1}{t}\int_0^t \mathbb{E}\|\mathcal{X}(s)\|_2^2\,ds
\le
\beta\,\|w\|_{\mathcal{L}_\infty}^2
+ \mathrm{tr}(N)$.
Since $\displaystyle\limsup_{t\to\infty}\mathbb{E}\|\mathcal{X}(t)\|_2^2
\le 
\limsup_{t\to\infty}
\frac{1}{t}\int_0^t \mathbb{E}\|\mathcal{X}(s)\|_2^2\,ds$,
it follows
\begin{gather}
\label{UltimateBound}
    \limsup_{t\to\infty} \mathbb{E}\|\mathcal{X}(t)\|_2^2
\le 
\frac{\beta\|w\|_{\mathcal{L}_\infty}^2 + \mathrm{tr}(N)}
{\alpha\,\lambda_{\min}(P)}.
\end{gather}
Thus, the closed-loop system is \emph{mean-square ultimately bounded}. Minimizing $\mathrm{tr}(\bm{N})$ minimizes this upper bound, proving the theorem. $\qquad\qquad\qquad\qquad\quad\qquad\qquad\quad$\qed                  
\end{proof}

\begin{remark}
\label{MatrixT}
	Matrix $\bm{T}$ defines a similarity transformation over the state of the dynamic controller, $\bm{x}_f${; thus, it} does not affect the closed-loop transfer function. The choice $\bm{T} = (\bm{S}-\bm{W}^{-1})$ is suitable {for mitigating} numerical issues.
\end{remark}

\vspace{-2mm}
\section{Illustrative numerical example}
\label{Results}
\vspace{-2mm}

To illustrate the proposed approach, we consider the F-16 fighter aircraft \GR{pitch-axis longitudinal dynamics} \citep{rajpurohit2017nonlinear}, modeled \GR{via} \eqref{GenericStochasticSystem} with
\small
 \begin{align*}
    \bm{A} &{=} \begin{bmatrix}
0 & 1 & 0 \\
0 & -0.87 & 43.22 \\
0 & 0.99 & -1.34 
\end{bmatrix}\!, ~
 \bm{B} {=}
\begin{bmatrix}
0 & 0 \\
-17.25 & -1.58 \\
-0.17 & -0.25
\end{bmatrix}\!, ~\bm{D} {=}\begin{bmatrix}
0 & 0.01 & 0 \\
0 & 0 & 0.1
\end{bmatrix}^T\!,  \\
\bm{A}_y &{=}\begin{bmatrix}
1 & 0 & 0 \\
0 & 1 & 0
\end{bmatrix},  ~~~
 \bm{E}_x {=} \text{diag}(0.001, 0.01, 0.05),   ~~~
 \bm{E}_y {=} \text{diag}(0.1, 0.1),
\end{align*}
\normalsize
with \GR{state} $\bm{x} = [{x}_1 \;\; {x}_2 \;\; {x}_3]^T$ \GR{(pitch angle, rate, and angle of attack)} and \GR{input} $\bm{u} = [{u}_1 \;\; {u}_2]^T$ \GR{(elevator and flaperon deflections)}. \GR{The} model \GR{represents} nominal flight at 3,000 feet and Mach 0.6 \GR{under} stochastic disturbances. \GR{Controllers were synthesized} by solving the SDPs in \eqref{OptLMIOutputFeedback} using YALMIP \citep{lofberg2004yalmip} and MOSEK \GR{in MATLAB}. \GR{Parameters} \GR{tuned to} $\bm{\Gamma}_0 = \text{diag}\big(1,1,1\big)$ and $\bm{\Gamma}_1 = \text{diag}\big(0.001, 0.001, 0.001\big)$. \GR{Applying the} control law \eqref{u1} \GR{to the system} \eqref{GenericStochasticSystem} \GR{with} initial conditions \GR{$\bm{x}_0 = [\pi/12, 0, \pi/12]^T$,} the resulting state trajectories, control inputs, and disturbances  are shown in Fig.~\ref{Results1}.

\vspace{-2mm}
\begin{figure}[htb!]
	\centering
	\def\svgwidth{0.9\columnwidth}{
			\scriptsize{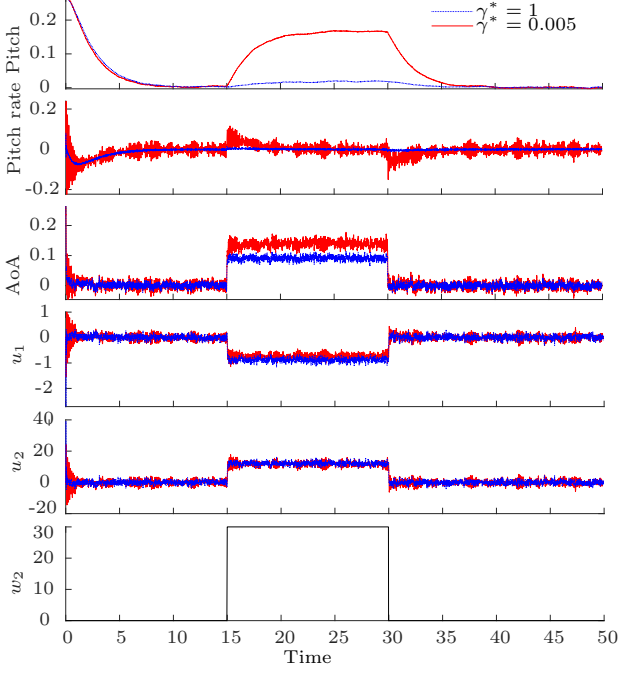}}
    \protect
    \caption{Time evolution of the pitch angle, pitch rate, and angle of attack (AoA), along with the control inputs $u_1$ and $u_2$ and the disturbance signal $w$, for two values of the $\mathcal{W}_\infty$-index, $\gamma^{*} \in \{1,\,0.005\}$.}

	\label{Results1}
\end{figure}

%As shown in the figure, the system starts from an initial displacement away from the origin and converges, in accordance with the dissipation inequality \eqref{DissipationInequality}, toward the mean-square ultimate bound \eqref{UltimateBound}. Because this convergence is in the mean-square sense, the trajectories naturally exhibit occasional increases during the transient phase due to stochastic fluctuations. A disturbance is applied to the angle-of-attack dynamics between 15 and 30 seconds. We simulated the closed-loop system for two values of the $\mathcal{W}_\infty$ index, $\gamma^* \in \{1, 0.05\}$. As discussed below equation \eqref{OptimalWinfControlProblem}, a smaller $\mathcal{W}_\infty$ index produces a controller with better disturbance attenuation and faster transient behavior. This property is particularly important in stochastic environments, where fluctuations can substantially influence the system trajectories. 

\vspace{-1mm}
As shown, the system converges from an initial displacement to the mean-square ultimate bound \eqref{UltimateBound} per \eqref{DissipationInequality}, with natural stochastic transient increases. An angle-of-attack disturbance is applied between 15 and 30 seconds. Simulations for $\gamma^* \in \{1, 0.05\}$ confirm that a smaller $\mathcal{W}_\infty$ index yields faster transients and superior disturbance attenuation (see below \eqref{OptimalWinfControlProblem}), a property critical for mitigating fluctuations in stochastic environments.

\vspace{-1mm}
\section{Conclusion}
\vspace{-2mm}

This work introduced a stochastic robust $\mathcal{W}_\infty$ optimal control approach in weighted Sobolev \GR{spaces}. \GR{Although} stochastic state trajectories are almost surely non-differentiable\GR{, we demonstrated} that the expected state admits a weak derivative, enabling a consistent $\mathcal{W}_\infty$ cost functional that \GR{penalizes both the} expected state and its \GR{rate of change}. \GR{Based on this,} we established the stochastic robust $\mathcal{W}_\infty$ OCP and derived the \GR{corresponding} HJ equation\GR{, which was reformulated as} an equivalent SDP \GR{for efficient} dynamic output-feedback controllers \GR{synthesis.} \GR{Stability} analysis \GR{confirmed} mean-square ultimate boundedness\GR{, with the proposed controller minimizing this bound.} A numerical experiment \GR{using} F-16 aircraft \GR{longitudinal dynamics demonstrated superior} attenuation of stochastic fluctuations and rapid mitigation. Future \GR{work} will investigate explicit HJ \GR{solutions} for stochastic nonlinear systems.

\vspace{-2mm}

\bibliography{ifacconf}             % bib file to produce the bibliography

\begin{thebibliography}{19}
\providecommand{\natexlab}[1]{#1}
\providecommand{\url}[1]{\texttt{#1}}
\providecommand{\urlprefix}{URL }
\expandafter\ifx\csname urlstyle\endcsname\relax
  \providecommand{\doi}[1]{doi:\discretionary{}{}{}#1}\else
  \providecommand{\doi}{doi:\discretionary{}{}{}\begingroup
  \urlstyle{rm}\Url}\fi

\bibitem[{Aliyu and Boukas(2011)}]{aliyu2011extending}
Aliyu, M.D.S. and Boukas, E.K. (2011).
\newblock Extending nonlinear {$\mathcal{H}_2$}, {$\mathcal{H}_\infty$}
  optimisation to {$\mathcal{W}_{1,2}$}, {$\mathcal{W}_{1,\infty}$} spaces -
  part {I}: optimal control.
\newblock \emph{International Journal of Systems Science}, 42(5), 889--906.

\bibitem[{Cardoso et~al.(2021)Cardoso, Esteban, and Raffo}]{AUT2019}
Cardoso, D.N., Esteban, S.R., and Raffo, G.V. (2021).
\newblock A robust optimal control approach in the weighted sobolev space for
  underactuated mechanical systems.
\newblock \emph{Automatica}, 125, 1--11.

\bibitem[{Chen(1984)}]{chen1984linear}
Chen, C.T. (1984).
\newblock \emph{Linear system theory and design}.
\newblock Saunders college publishing.

\bibitem[{Chilali and Gahinet(1996)}]{chilali1996h}
Chilali, M. and Gahinet, P. (1996).
\newblock $\mathcal{H}_\infty$ design with pole placement constraints: An {LMI}
  approach.
\newblock \emph{Transactions on automatic control}, 41(3), 358--367.

\bibitem[{Knabner and Angermann(2003)}]{knabner2003numerical}
Knabner, P. and Angermann, L. (2003).
\newblock \emph{Numerical methods for elliptic and parabolic partial
  differential equations}.
\newblock Springer.

\bibitem[{Lofberg(2004)}]{lofberg2004yalmip}
Lofberg, J. (2004).
\newblock Yalmip: A toolbox for modeling and optimization in matlab.
\newblock In \emph{2004 IEEE international conference on robotics and
  automation}, 284--289. IEEE.

\bibitem[{Ma et~al.(2012)Ma, Zhang, and Hou}]{ma2012infinite}
Ma, H., Zhang, W., and Hou, T. (2012).
\newblock Infinite horizon $\mathcal{H}_2/\mathcal{H}_\infty$ control for
  discrete-time time-varying markov jump systems with multiplicative noise.
\newblock \emph{Automatica}, 48(7), 1447--1454.

\bibitem[{Peters and Stoorvogel(1994)}]{peters1994mixed}
Peters, M. and Stoorvogel, A. (1994).
\newblock Mixed $\mathcal{H}_2/\mathcal{H}_\infty$ control in a stochastic
  framework.
\newblock \emph{Linear algebra and its applications}, 205, 971--996.

\bibitem[{Rajpurohit and Haddad(2017)}]{rajpurohit2017nonlinear}
Rajpurohit, T. and Haddad, W.M. (2017).
\newblock Nonlinear--nonquadratic optimal and inverse optimal control for
  stochastic dynamical systems.
\newblock \emph{International Journal of Robust and Nonlinear Control}, 27(18),
  4723--4751.

\bibitem[{Schaft(1992)}]{van1992sub}
Schaft, A.V.D. (1992).
\newblock $\mathcal{L}_2$-gain analysis of nonlinear systems and nonlinear
  state-feedback $\mathcal{H}_\infty$ control.
\newblock \emph{Transactions on Automatic Control}, 37(6), 770--784.

\bibitem[{Sedhom et~al.(2020)Sedhom, El-Saadawi, Hatata, and
  Abd-Raboh}]{sedhom2020multistage}
Sedhom, B.E., El-Saadawi, M.M., Hatata, A.Y., and Abd-Raboh, E.E. (2020).
\newblock A multistage h-infinity-based controller for adjusting voltage and
  frequency and improving power quality in islanded microgrids.
\newblock \emph{International Transactions on Electrical Energy Systems},
  30(1), 1--34.

\bibitem[{Shreve et~al.(2004)}]{shreve2004stochastic}
Shreve, S.E. et~al. (2004).
\newblock \emph{Stochastic calculus for finance II: Continuous-time models},
  volume~11.
\newblock Springer.

\bibitem[{Treves(2016)}]{treves2016topological}
Treves, F. (2016).
\newblock \emph{Topological Vector Spaces, Distributions and Kernels: Pure and
  Applied Mathematics}, volume~25.
\newblock Elsevier.

\bibitem[{Ugrinovskii(1998)}]{ugrinovskii1998robust}
Ugrinovskii, V.A. (1998).
\newblock Robust $\mathcal{H}_\infty$ infinity control in the presence of
  stochastic uncertainty.
\newblock \emph{International Journal of Control}, 71(2), 219--237.

\bibitem[{Wu et~al.(2011)Wu, Zhang, and Chen}]{wu2011multiobjective}
Wu, C.H., Zhang, W., and Chen, B.S. (2011).
\newblock Multiobjective $\mathcal{H}_2/\mathcal{H}_\infty$ synthetic gene
  network design based on promoter libraries.
\newblock \emph{Mathematical Biosciences}, 233(2), 111--125.

\bibitem[{Yong and Zhou(2012)}]{yong2012stochastic}
Yong, J. and Zhou, X.Y. (2012).
\newblock \emph{Stochastic controls: Hamiltonian systems and HJB equations},
  volume~43.
\newblock Springer Science \& Business Media.

\bibitem[{Zhang and Chen(2006)}]{zhang2006state}
Zhang, W. and Chen, B.S. (2006).
\newblock State feedback $\mathcal{H}_\infty$ control for a class of nonlinear
  stochastic systems.
\newblock \emph{SIAM journal on control and optimization}, 44(6), 1973--1991.

\bibitem[{Zhang et~al.(2017)Zhang, Xie, and Chen}]{zhang2017stochastic}
Zhang, W., Xie, L., and Chen, B.S. (2017).
\newblock \emph{Stochastic $\mathcal{H}_2$/$\mathcal{H}_\infty$ control: A Nash
  game approach}.
\newblock CRC Press.

\bibitem[{Zhu et~al.(2018)Zhu, Suo, Chen, Zhang, and Li}]{zhu2018mixed}
Zhu, B., Suo, M., Chen, Y., Zhang, Z., and Li, S. (2018).
\newblock Mixed $\mathcal{H}_\infty$ and passivity control for a class of
  stochastic nonlinear sampled-data systems.
\newblock \emph{Journal of the Franklin Institute}, 355(7), 3310--3329.

\end{thebibliography}
                                                     % with bibtex (preferred)

\end{document}